\documentclass[reqno,10pt, centertags, draft]{amsart}
\usepackage{amsmath,amsthm,amscd,latexsym}
\usepackage{amssymb,upref,esint,dsfont,cite,amsmath}
\usepackage[dvipsnames]{xcolor}
\usepackage{hyperref}
\newcommand*{\mailto}[1]{\href{mailto:#1}{\nolinkurl{#1}}}
\newcommand{\arxiv}[1]{\href{http://arxiv.org/abs/#1}{arXiv:#1}}

\makeatletter
\def\theequation{\@arabic\c@equation}

\newcommand{\bbN}{{\mathbb{N}}}
\newcommand{\bbR}{{\mathbb{R}}}

\newcommand{\bbC}{{\mathbb{C}}}

\newcommand{\bbS}{{\mathbb{S}}}

\newcommand{\cX}{{\mathcal X}}
\newcommand{\cY}{{\mathcal Y}}

\newcommand{\dott}{\,\cdot\,}

\newcommand{\no}{\nonumber}
\newcommand{\lb}{\label}
\newcommand{\f}{\frac}

\newcommand{\ol}{\overline}
\newcommand{\bs}{\backslash} 

\newcommand{\wti}{\widetilde}
\newcommand{\eps}{\varepsilon}

\newcommand{\al}{\alpha}

\renewcommand{\b}{\beta}
\newcommand{\g}{\gamma}

\newcommand{\loc}{\text{\rm{loc}}}

\newcommand{\dom}{\text{\rm{dom}}}

\newcommand{\supp}{\text{\rm{supp}}}

\newcommand{\bi}{\bibitem}

\newcommand{\hatt}{\widehat}
\renewcommand{\Re}{\text{\rm Re}}

\renewcommand{\ln}{\text{\rm ln}}

\renewcommand{\dot}{\overset{\textbf{\Large.}}}

\DeclareMathOperator{\Div}{div}

\numberwithin{equation}{section}

\newtheorem{theorem}{Theorem}[section]
\newtheorem{lemma}[theorem]{Lemma}

\newtheorem{definition}[theorem]{Definition}
\newtheorem{hypothesis}[theorem]{Hypothesis}

\theoremstyle{remark}
\newtheorem{remark}[theorem]{Remark}

\begin{document}

\title[Essential Self-Adjointness of Sturm--Liouville Operators]{On Essential Self-Adjointness of Singular \\ Sturm--Liouville Operators}
%{The $A$-Equation for Dirac-Type Operators}

\author[S.\ B.\ Allan]{S. Blake Allan}
\address{Department of Mathematics, 
Baylor University, Sid Richardson Bldg., 1410 S.\,4th Street,
Waco, TX 76706, USA}
\email{\mailto{Blake\_Allan1@baylor.edu}}
%\email{Blake$\_$Allan1@baylor.edu}

\author[F.\ Gesztesy]{Fritz Gesztesy}
\address{Department of Mathematics,
Baylor University, Sid Richardson Bldg., 1410 S.\,4th Street,
Waco, TX 76706, USA}
\email{\mailto{Fritz\_Gesztesy@baylor.edu}}
%\email{Fritz$\_$Gesztesy@baylor.edu}
\urladdr{\url{http://www.baylor.edu/math/index.php?id=935340}}
%\urladdr{http://www.baylor.edu/math/index.php?id=935340}

\author[A.\ L.\ Sakhnovich]{Alexander Sakhnovich}
\address{Faculty of Mathematics\\ University of Vienna\\
Oskar-Morgenstern-Platz 1\\ 1090 Wien\\ Austria}
\email{\mailto{oleksandr.sakhnovych@univie.ac.at}}
%\email{oleksandr.sakhnovych@univie.ac.at}
\urladdr{\url{http://www.mat.univie.ac.at/~sakhnov/}} 
%\urladdr{http://www.mat.univie.ac.at/~sakhnov/}

%%%%%%%%%%%%%%%%%%%%%%%%%%%%%%%%%%%%%%%%% 
%\dedicatory{}
\date{\today}
%\date{, 2003.}
%\thanks{.} 
\thanks{To appear in {\it Revista de la Unión Matemática Argentina}  (RevUMA).}
\@namedef{subjclassname@2020}{\textup{2020} Mathematics Subject Classification}

%%%%%%%%%%%%%%%%%%%%%%%%%%%%%%%%%%%%%%%%% 
\begin{abstract}
Considering singular Sturm--Liouville differential expressions of the type 
\[
\tau_{\alpha} = -(d/dx)x^{\alpha}(d/dx) + q(x), \quad x \in (0,b),  \; \alpha \in \bbR, 
\]
we employ some Sturm comparison-type results in the spirit of Kurss to derive criteria for $\tau_{\alpha}$ to be in the limit point and limit circle case at $x=0$.
More precisely, if $\alpha \in \bbR$ and for $0 < x$ sufficiently small, 
\[
q(x) \geq [(3/4)-(\alpha/2)]x^{\alpha-2},
\]
or, if $\alpha\in (-\infty,2)$ and there exist $N\in\bbN$, and $\varepsilon>0$ such 
that for $0<x$ sufficiently small, 
\begin{align*}
&q(x)\geq[(3/4)-(\alpha/2)]x^{\alpha-2} - (1/2) (2 - \alpha) x^{\alpha-2} 
\sum_{j=1}^{N}\prod_{\ell=1}^{j}[\ln_{\ell}(x)]^{-1}     \\
&\quad\quad\quad +[(3/4)+\varepsilon] x^{\alpha-2}[\ln_{1}(x)]^{-2}. 
\end{align*}
then $\tau_{\alpha}$ is nonoscillatory and in the limit point case at $x=0$. Here iterated logarithms for $0 < x$ sufficiently small are of the form, 
\[
\ln_1(x) = |\ln(x)| = \ln(1/x), \quad \ln_{j+1}(x) = \ln(\ln_j(x)), \quad j \in \bbN.
\]
Analogous results are derived for $\tau_{\alpha}$ to be in the limit circle case at $x=0$. 

We also discuss a multi-dimensional application to partial differential expressions of the type   
\[
- \Div |x|^{\alpha} \nabla + q(|x|), \quad \alpha \in \bbR, \; x \in B_n(0;R) \backslash\{0\},
\]
with $B_n(0;R)$ the open ball in $\bbR^n$, $n\in \bbN$, $n \geq 2$, centered at $x=0$ of radius $R \in (0, \infty)$. 
\end{abstract}
%%%%%%%%%%%%%%%%%%%%%%%%%%%%%%%%%%%%%%%%% 

\maketitle

%\newpage 

{\scriptsize{\tableofcontents}}
%\normalsize

%%%%%%%%%%%%%%%%%%%%%%%%%%%%%%%%%%%%%%%%%
%%%%%%%%%%%%%%%%%%%%%%%%%%%%%%%%%%%%%%%%%
\section{Introduction} \lb{s1}
%%%%%%%%%%%%%%%%%%%%%%%%%%%%%%%%%%%%%%%%%
%%%%%%%%%%%%%%%%%%%%%%%%%%%%%%%%%%%%%%%%%

In a nutshell, we are interested in deriving limit point and limit circle criteria for singular differential expressions of the type 
\begin{equation}
\tau_{\alpha} = -(d/dx)x^{\alpha}(d/dx) + q(x), \quad x \in (0,b),  \; \alpha \in \bbR.      \lb{1.1}
\end{equation}
The principal results we prove in Theorems \ref{t3.1} and \ref{t3.3} on the basis of Sturm comparison-type results initiated by Kurss are the following: Suppose $\alpha \in \bbR$ and for $0 < x$ sufficiently small, 
\begin{equation}
q(x) \geq [(3/4)-(\alpha/2)]x^{\alpha-2},     \lb{1.2} 
\end{equation}
or, in the context of logarithmic refinements of \eqref{1.2}, assume that $\alpha\in (-\infty,2)$ and there exist 
$N\in\bbN$ and $\varepsilon>0$ such that for $0<x$ sufficiently small,  
\begin{align}
\begin{split} 
&q(x)\geq[(3/4)-(\alpha/2)]x^{\alpha-2} - (1/2) (2 - \alpha) x^{\alpha-2} 
\sum_{j=1}^{N}\prod_{\ell=1}^{j}[\ln_{\ell}(x)]^{-1}  \lb{1.3}   \\
&\quad\quad\quad +[(3/4)+\varepsilon] x^{\alpha-2}[\ln_{1}(x)]^{-2}. 
\end{split} 
\end{align}
(For the definition of the iterated logarithms $\ln_{\ell}(\dott)$, $\ell \in \bbN$, we refer to \eqref{3.3}.)
Then, in either situation \eqref{1.2} and \eqref{1.3}, $\tau_{\alpha}$ is nonoscillatory and in the limit point case at $x=0$.

Similarly, if $\alpha \in (-\infty,2)$ and there exists $\varepsilon \in (0,1)$ $($depending on $\alpha)$ such that 
for $0<x$ sufficiently small $($depending on $\varepsilon)$, 
\begin{equation}
q(x)\leq[(3/4)-(\alpha/2) - \varepsilon]x^{\alpha-2},     \lb{1.4}
\end{equation}
or, if $\alpha\in (-\infty,2)$ and there exist $N\in\bbN$ and $\varepsilon\in(0,1)$ $($depending on 
$\alpha$ and $N$$)$, such that 
for $0<x$ sufficiently small $($depending on $\varepsilon$ and $N$$)$, 
\begin{align}
q(x)&\leq[(3/4)-(\alpha/2)]x^{\alpha-2} 
- (1/2) (2 - \alpha) x^{\alpha-2}\sum_{j=1}^{N}\prod_{\ell=1}^{j}[\ln_{\ell}(x)]^{-1}\no\\
&\quad - [\varepsilon(2 - \alpha)/2] x^{\alpha-2}\prod_{k=1}^{N}[\ln_{k}(x)]^{-1}.
	\lb{1.5}
\end{align}
Then, in either situation \eqref{1.4} and \eqref{1.5}, $\tau_{\alpha}$ is in the limit circle case at $x=0$.

We note that the amount of literature on limit point/limit circle criteria for Sturm--Liouville operators is so immense that there is no possibility to account for it in this short paper. The reader finds very thorough discussions of this circle of ideas in  
\cite[Ch.~XIII.6, Sect.~XIII.10.D]{DS88}, \cite[Ch.~4.6]{GZ20}, \cite[Sect.~23.6]{Na68}, \cite[App.~to Sect.~X.1]{RS75}, \cite[Sect.~13.4]{We03}, \cite[Ch.~7]{Ze05}, and the extensive literature cited therein. One of the driving motivations for  writing this paper was to emphasize the simplicity of proofs of Theorems \ref{t3.1} and \ref{t3.3} (essentially, they are reduced to certain computations), given the Sturm comparison results, Theorems \ref{t2.8} and \ref{t2.9}, in the spirit of Kurss. 

We briefly turn to the content of each section next: Section \ref{s2} provides the necessary background on minimal and maximal operators associated with general three-coefficient Sturm--Liouville differential expressions, recalls Weyl's limit point/limit circle classification and some oscillation theory, as well as Sturm comparison results, quoting some extensions of the classical work by Kurss \cite{Ku67}.  Section \ref{s3} contains our principal results, Theorems \ref{t3.1} and \ref{t3.3}, as summarized in \eqref{1.2}--\eqref{1.5}. Section \ref{s4} considers an elementary multi-dimensional application to partial differential expressions of the type   
\[
- \Div |x|^{\alpha} \nabla + q(|x|), \quad \alpha \in \bbR, \; x \in B_n(0;R) \backslash\{0\},
\]
with $B_n(0;R)$ the open ball in $\bbR^n$, $n\in \bbN$, $n \geq 2$, centered at $x=0$ of radius $R \in (0, \infty)$. 
Finally, Appendix \ref{sA} contains some more involved computations needed in the proofs of Theorems \ref{t3.1} and \ref{t3.3}.

%%%%%%%%%%
%%%%%%%%%%
\section{Some Background on Sturm--Liouville Operators} \lb{s2}
%%%%%%%%%%
%%%%%%%%%%

In this section we briefly recall the necessary background on maximal and minimal operators corresponding to three-coefficient Sturm--Liouville differential expressions, introduce the notion of deficiency indices for symmetric Sturm--Liouville operators, discuss Weyl's limit point/limit circle dichotomy, recall some oscillation theory, and discuss an extension of Sturm comparison results, applicable to the limit point/limit circle case, following Kurss \cite{Ku67}. Standard references for much of this material are, for instance, \cite[Ch.~6]{{BHS20}}, 
\cite[Chs.~8, 9]{CL85}, \cite[Sects~13.6, 13.9, 13.0]{DS88}, \cite[Ch.~4]{GZ20}, \cite[Ch.~III]{JR76}, 
\cite[Ch.~V]{Na68}, \cite{NZ92}, \cite[Ch.~6]{Pe88}, \cite[Ch.~9]{Te14}, \cite[Sect.~8.3]{We80}, \cite[Ch.~13]{We03}, \cite[Chs.~4, 6--8]{Ze05}.

We start with the basic set of assumptions throughout this section:

%%%%%%%%%%%%%
\begin{hypothesis} \lb{h2.1}
Let $(a,b) \subseteq \bbR$ and suppose that $p,q,r$ are $($Lebesgue\,$)$ measurable functions on $(a,b)$ 
such that the following items $(i)$--$(iii)$ hold: \\[1mm] 
$(i)$ \hspace*{1.1mm} $r>0$ a.e.~on $(a,b)$, $r \in L^1_{loc}((a,b); dx)$. \\[1mm] 
$(ii)$ \hspace*{.1mm} $p>0$ a.e.~on $(a,b)$, $1/p \in L^1_{loc}((a,b); dx)$. \\[1mm] 
$(iii)$ $q$ is real-valued a.e.~on $(a,b)$, $q \in L^1_{loc}((a,b); dx)$. 
\end{hypothesis}
%%%%%%%%%%%%%

Given Hypothesis \ref{h2.1}, we briefly study Sturm--Liouville differential expressions $\tau$ of the type,
\begin{equation}
\tau=\f{1}{r(x)}\left[-\f{d}{dx}p(x)\f{d}{dx} + q(x)\right] \, \text{ for a.e.~$x\in(a,b) \subseteq \bbR$}   \lb{2.1}
\end{equation} 

Given $\tau$ as in \eqref{2.1}, the corresponding {\it maximal operator} $T_{max}$ in $L^2((a,b); r\,dx)$ associated with $\tau$ is defined by
\begin{align}
\begin{split}
&T_{max} f = \tau f,\\
& f \in \dom(T_{max})=\big\{g\in L^2((a,b); r\,dx)\, \big| \,g,g^{[1]} \in AC_{loc}((a,b));  \\ 
& \hspace*{6.3cm}  \tau g\in L^2((a,b); r\,dx)\big\},
\end{split}
	\lb{2.2}
\end{align}
with $g^{[1]} = p g'$ denoting the first quasi-derivative of $g$. 
The {\it preminimal operator} $\dot T_{min} $ in $L^2((a,b); r\,dx)$ associated with $\tau$ is defined by 
\begin{align}
\begin{split}
&\dot T_{min}  f = \tau f,
\\
&f \in \dom \big(\dot T_{min}\big)=\big\{g\in L^2((a,b); r\,dx) \, \big| \, g,g^{[1]}\in AC_{loc}((a,b));
\\
&\hspace*{3.25cm} \supp \, (g)\subset(a,b) \text{ is compact; } \tau g\in L^2((a,b); r\,dx)\big\}.
\end{split}
	\lb{2.3}
\end{align}
One can prove that $\dot T_{min} $ is closable and then define the {\it minimal operator} $T_{min}$ in 
$L^2((a,b); rdx)$ as the closure of $\dot T_{min}$, 
\begin{equation}
T_{min} = \ol{\dot T_{min}}.
	\lb{2.4}
\end{equation}

The following result recalls Weyl's celebrated alternative:

%%%%%%%%%%%%%
\begin{theorem} \lb{t2.2} ${}$ \\
Assume Hypothesis \ref{h2.1}. Then the following alternative holds: \\[1mm] 
$(i)$ For every $z\in\bbC$, all solutions $u$ of $(\tau-z)u=0$ are in $L^2((a,b); r\,dx)$ near $b$ 
$($resp., near $a$$)$. \\[1mm] 
$(ii)$  For every $z\in\bbC$, there exists at least one solution $u$ of $(\tau-z)u=0$ which is not in $L^2((a,b); r\,dx)$ near $b$ $($resp., near $a$$)$. In this case, for each $z\in\bbC\bs\bbR$, there exists precisely one solution $u_b$ $($resp., $u_a$$)$ of $(\tau-z)u=0$ $($up to constant multiples\,$)$ which lies in $L^2((a,b); rdx)$ near $b$ $($resp., near $a$$)$. 
\end{theorem}
%%%%%%%%%%

This naturally leads to the notion that $\tau$ is in the limit point or limit circle case at an interval endpoint as follows:

%%%%%%%%%%%%%
\begin{definition} \lb{d2.3} Assume Hypothesis \ref{h2.1}. \\[1mm]  
In case $(i)$ in Theorem \ref{t2.2}, $\tau$ is said to be in the \textit{limit circle case} at $b$ $($resp., at $a$$)$. 
\\[1mm] 
In case $(ii)$ in Theorem \ref{t2.2}, $\tau$ is said to be in the \textit{limit point case} at $b$ $($resp., at $a$$)$. 
\end{definition}
%%%%%%%%%%%%%

The deficiency indices of $T_{min}$ are then given by 
\begin{align}
n_\pm(T_{min}) &= \dim(\ker(T_{max} \mp i I))    \no\\
\begin{split}
& = \begin{cases}
2 & \text{if $\tau$ is in the limit circle case at $a$ and $b$,}\\
1 & \text{if $\tau$ is in the limit circle case at $a$} \\
& \text{and in the limit point case at $b$, or vice versa,}\\
0 & \text{if $\tau$ is in the limit point case at $a$ and $b$}. 
\end{cases}
\end{split}
	\lb{2.5}
\end{align}
In particular, $T_{min} = T_{max}$ is self-adjoint \big(i.e., $\dot T_{min}$ is essentially self-adjoint\big) 
if and only if $\tau$ is in the limit point case at $a$ and $b$, underscoring the special role played by limit point endpoints (as opposed to a limit circle endpoint that requires a boundary condition in connection with self-adjointness issues of $T_{min}$).  

We continue with a few remarks on Sturm's oscillation theory (see, e.g., \cite[Theorem~7.4.4]{GZ20}, \cite[Sect.~14]{We87}, and the detailed list of references cited therein): 

%%%%%%%%%%%%%%
\begin{definition} \lb{d2.4}
Assume Hypothesis \ref{h2.1}. \\[1mm] 
$(i)$ Fix $c\in (a,b)$ and $\lambda\in\bbR$. Then $\tau - \lambda$ is
called {\it nonoscillatory} at $a$ $($resp., $b)$, 
if there exists a real-valued solution $u(\lambda,\dott)$ of 
$\tau u = \lambda u$ that has finitely many
zeros in $(a,c)$ $($resp., $(c,b))$. Otherwise, $\tau - \lambda$ is called {\it oscillatory}
at $a$ $($resp., $b)$. \\[1mm] 
$(ii)$ Let $\lambda_0 \in \bbR$. Then $T_{min}$ is called bounded from below by $\lambda_0$, 
and one writes $T_{min} \geq \lambda_0 I$, if 
\begin{equation} 
(u, [T_{min} - \lambda_0 I]u)_{L^2((a,b);r\,dx)}\geq 0, \quad u \in \dom(T_{min}).
	\lb{2.6}
\end{equation}
\end{definition}
%%%%%%%%%%%%%%

%%%%%%%
\begin{remark} \lb{r2.5}
By Sturm's separation theorem, $\tau - \lambda$, $\lambda\in\bbR$, is nonoscillatory at $a$ (resp., at $b$) 
if and only if every real-valued solution $u(\lambda,\dott)$ of $\tau u = \lambda u$ has finitely many zeros in 
$(a,c)$ (resp., $(c,b)$).
\hfill $\diamond$
\end{remark}
%%%%%%%

The following is a key result relating the notions of boundedness from below and nonoscillation. 

%%%%%%%%%%%%%%
\begin{theorem} \lb{t2.6} 
Assume Hypothesis \ref{h2.1}. Then the following items $(i)$ and $(ii)$ are
equivalent: \\[1mm] 
$(i)$ $T_{min}$ $($and hence any symmetric extension of $T_{min})$
is bounded from below. \\[1mm] 
$(ii)$ There exists a $\nu_0\in\bbR$ such that for all $\lambda < \nu_0$, $\tau - \lambda$ is
nonoscillatory at $a$ and $b$. 
\end{theorem}
%%%%%%%%%%%%%%

We also recall Sturm's comparison result in the following form.

%%%%%
\begin{theorem} \lb{t2.7} 
Assume that $p, q_j, r$ satisfy Hypothesis \ref{h2.1} and denote
\begin{equation} 
\tau_j = r(x)^{-1}[-(d/dx) p(x) (d/dx) + q_j(x)]  \, \text{ for a.e.~$x\in(a,b) \subseteq \bbR$,} \; j =1,2.  
	\lb{2.7}
\end{equation}  
Fix $\lambda \in \bbR$ and let $u_j$ be a real-valued solution of $\tau_j u_j= \lambda u_j$, $j=1,2$. 
If for some $x_0\in (a,b) \subseteq \bbR$,  
\begin{align}
\begin{split}
& q_2\geq q_1 \, \text{ a.e.\ on } (a,b),
\\
& u_1\neq 0 \, \text{ on } (a,b)\bs\{x_0\},
\\
& u_2(x_0)=u_1(x_0),\quad u_2^{[1]}(x_0)=u_1^{[1]}(x_0),
\end{split}
	\lb{2.8}
\end{align}
then
\begin{align}
|u_2(x)|\geq|u_1(x)| \, \text{ for all } x\in(a,b),
	\lb{2.9}
\end{align}
in particular,
\begin{equation}
 u_2 \neq 0 \, \text{ on } (a,b)\bs\{x_0\}. 
	\lb{2.10}
\end{equation}
\end{theorem}
%%%%%%%%

Next, we also recall the following general comparison result (due to Kurss \cite{Ku67} in the special 
case\footnote{One of us, F.G., is indebted to Hubert Kalf \cite{Ka15} for kindly and repeatedly sharing his detailed notes on the proof of Theorem \ref{t2.8} and on various ramifications of this circle of ideas.} $r=1$):

%%%%%%%%
\begin{theorem} \lb{t2.8} 
Assume that $p, q_j, r$ satisfy Hypothesis \ref{h2.1} and denote 
\begin{equation}
\tau_j = r(x)^{-1}[-(d/dx) p(x) (d/dx) + q_{j}(x)] \, \text{ for a.e.~$x\in(a,b) \subseteq \bbR$,} \; j =1,2.
     \lb{2.11}
\end{equation}
Suppose that $\tau_1$ is nonoscillatory and in the limit point case at $a$ and 
that $q_2 \geq q_1$ a.e.\ on $(a,b) \subseteq \bbR$. Then $\tau_2$ is also nonoscillatory and 
in the limit point case at $a$. The analogous statement applies to the endpoint $b$. 
\end{theorem}
%%%%%%%%

Theorem \ref{t2.8} is a consequence of Theorem \ref{t2.7} (cf. \cite[Theorem~7.4.6]{GZ20}), as is its limit circle analogue below:

%%%%%%%%
\begin{theorem}	\lb{t2.9}
Under the assumptions of Theorem \ref{t2.8}, suppose $\tau_{1}$ is nonoscillatory and in the limit circle case at $a$ and that $q_{2}\leq q_{1}$ a.e.~on $(a,b)\subseteq\bbR$.  Then $\tau_{2}$ is in the limit circle case at $a$.  The analogous statement applies to the endpoint $b$.
\end{theorem}
%%%%%%%%
\begin{proof}
By assumption, both fundamental solutions $u_{1,1}$ and $u_{1,2}$ of $\tau_{1}u_{1}=0$ are strictly positive on $(a,x_{0})$ for some $x_{0}\in(a,b)$ with $u_{1,j}\in L^{2}((a,x_{0});r\,dx)$, $j\in\{1,2\}$.  Let $u_{2,j}$, $j\in\{1,2\}$ be the solutions of $\tau_{2}u_{2}=0$ with initial data $u_{2,j}(x_{0})=u_{1,j}(x_{0})$ and $u_{2,j}^{[1]}(x_{0})=u_{1,j}^{[1]}(x_{0})$.  Then, by Theorem \ref{t2.7}, $|u_{2,j}(x)|\leq|u_{1,j}(x)|$ for all $x\in(a,x_{0})$, which implies $u_{2,j}\in L^{2}((a,x_{0});r\,dx)$, $j\in\{1,2\}$.  Thus $\tau_{2}$ is in the limit circle case at $a$.
\end{proof}
%%%%%%%%

%%%%%%%%%%
%%%%%%%%%%
\section{Results on the Limit Point and Limit Circle Case} \lb{s3}
%%%%%%%%%%
%%%%%%%%%%

In this section we prove our principal limit point/limit circle results on the special two-coefficient differential expression 
$\tau_{\alpha}$ in \eqref{3.1} below at $x=0$. 

Specializing to the case of half-line two-coefficient Sturm--Liouville operators, where 
\begin{align} 
\begin{split} 
& a = 0, \quad b \in (0,\infty) \cup \{\infty\}, \quad p(x)=x^{\alpha}, \; \alpha \in \bbR, \quad r(x)=1, \; x \in (0,b), \\
& q \in L_{loc}^{1}((0,b);dx) \, \text{ is real-valued~a.e.},
\end{split}
	\lb{3.1}
\end{align} 
$\tau$ in \eqref{2.1} now takes on the simplified form, 
\begin{equation}
\tau_{\alpha} = -(d/dx)x^{\alpha}(d/dx) + q(x) \, \text{ for a.e.~$x \in (0,b)$,  $\alpha \in \bbR$.}
	\lb{3.2}
\end{equation}
Given $\alpha \in \bbR$, or $\alpha \in (-\infty,2)$, we will derive conditions on $q$ that imply the limit point (and nonoscillatory) or limit circle behavior of $\tau_{\alpha}$ at $x=0$.  

Introducing iterated logarithms for $0 < x$ sufficiently small,
\begin{equation}
\ln_1(x) = |\ln(x)| = \ln(1/x), \quad \ln_{j+1}(x) = \ln(\ln_j(x)), \quad j \in \bbN
	\lb{3.3}
\end{equation}
(rendering $\ln_{\ell}(\dott)$, $1 \leq \ell \leq N$, in Theorems \ref{t3.1} and \ref{t3.3} below, strictly positive), our first principal result reads as follows:

%%%%%%%%
\begin{theorem} \lb{t3.1}
Suppose that $q \in L_{loc}^{1}((0,b);dx)$ is real-valued a.e.~on $(0,b)$. \\[1mm]
$(i)$ Let $\alpha \in \bbR$ and assume that for a.e.~$0<x$ sufficiently small, 
\begin{equation}
q(x) \geq [(3/4)-(\alpha/2)]x^{\alpha-2}.
	\lb{3.4} 
\end{equation}
Then $\tau_{\alpha}$ is nonoscillatory and in the limit point case at $x=0$. \\[1mm]
$(ii)$ Let\footnote{Only $\alpha \in (- \infty,2)$ can improve on item $(i)$.} $\alpha\in (-\infty,2)$ and assume there exist $N\in\bbN$ and $\varepsilon>0$,  such 
that for a.e.~$0<x$ sufficiently small $($depending on $N$ and $\varepsilon$$)$,
\begin{align}
\begin{split} 
&q(x)\geq[(3/4)-(\alpha/2)]x^{\alpha-2} - (1/2) (2 - \alpha) x^{\alpha-2} 
\sum_{j=1}^{N}\prod_{\ell=1}^{j}[\ln_{\ell}(x)]^{-1}     \\
&\quad\quad\quad +[(3/4)+\varepsilon] x^{\alpha-2}[\ln_{1}(x)]^{-2} \equiv Q_{\alpha,N,\varepsilon}(x).
\end{split}       \lb{3.5}
\end{align}
Then $\tau_{\alpha}$ is nonoscillatory and in the limit point case at $x=0$.
\end{theorem}
%%%%%%%%
\begin{proof}
In the following $0 < x$ is assumed to be sufficiently small. \\[1mm] 
$(i)$ Abbreviating for $\alpha \in \bbR$,  
\begin{align}
& q_{\alpha,0}(x)= [(3/4) - (\alpha/2)]x^{\alpha-2},     \lb{3.6} \\
& y_{0}(x)=x^{-1/2},     \lb{3.7} \\
&\tau_{\alpha,0}=-(d/dx)x^{\alpha}(d/dx)+q_{\alpha,0}(x),
	\lb{3.8}
\end{align}
one confirms that 
\begin{align}
(\tau_{\alpha,0}\,y_{0})(x)=0, \quad \alpha \in \bbR.
	\lb{3.9}
\end{align}
In particular, \eqref{3.7} and \eqref{3.9} prove that $\tau_{\alpha,0}$, $\alpha \in \bbR$, is nonoscillatory at $x=0$. 
Moreover, since for $R>0$,
\begin{align}
y_{0}\notin L^{2}((0,R);dx),  \quad \alpha \in \bbR,
	\lb{3.10} 
\end{align}
$\tau_{\alpha,0}$, $\alpha \in \bbR$, is nonoscillatory and in the limit point case at $x=0$, and hence so is $\tau_{\alpha}$, 
$\alpha \in \bbR$, by Theorem \ref{t2.8}. \\[1mm] 
$(ii)$ Since the sum of the 2nd and 3rd terms on the right-hand side of \eqref{3.5} would be nonnegative for 
$\alpha \geq 2$, Theorem \ref{t2.8} yields that only the case $\alpha \in (-\infty, 2)$ can improve upon item $(i)$. Next, we abbreviate for $\alpha \in (-\infty,2)$, $N\in\bbN$, and $0 < x$ sufficiently small,\footnote{\,If $N=1$ one interprets, as usually, sums and products over empty index sets as $0$ and $1$, respectively. \lb{f1}}
\begin{align}
&q_{\alpha,N}(x)= [(3/4) - (\alpha/2)]x^{\alpha-2}  
- (1/2) (2 - \alpha) x^{\alpha-2}\sum_{j=1}^{N}\prod_{\ell=1}^{j}[\ln_{\ell}(x)]^{-1}\no\\
&\hspace{1.6cm}+(3/4)x^{\alpha-2}\sum_{j=1}^{N}\prod_{\ell=1}^{j}[\ln_{\ell}(x)]^{-2}\no\\
&\hspace{1.6cm}+x^{\alpha-2}\sum_{j=1}^{N-1}\prod_{\ell=1}^{j}[\ln_{\ell}(x)]^{-2}\sum_{m=j+1}^{N}\prod_{p=j+1}^{m}[\ln_{p}(x)]^{-1}  , 	\lb{3.11}\\
&y_{N}(x)=x^{-1/2}\prod_{k=1}^{N} [\ln_{k}(x)]^{-1/2},
	\lb{3.12}\\
&\tau_{\alpha,N}=-(d/dx)x^{\alpha}(d/dx)+q_{\alpha,N}(x), 
	\lb{3.13}
\end{align}
and claim (cf.\ Lemma \ref{lA.1}) that 
\begin{align}
(\tau_{\alpha,N}\,y_{N})(x)=0, \quad \alpha \in (-\infty,2), \; N \in \bbN.
	\lb{3.14}
\end{align}
Once more, \eqref{3.12} and \eqref{3.14} prove that $\tau_{\alpha,N}$, $\alpha \in (-\infty,2)$, $N \in \bbN$, is nonoscillatory at $x=0$. Moreover, since for $0 < \delta_N$ sufficiently small,
\begin{align}
y_{N}\notin L^{2}((0,\delta_N);dx),
	\lb{3.15}
\end{align}
$\tau_{\alpha,N}$, $\alpha \in (-\infty,2)$, $N \in \bbN$, is nonoscillatory and in the limit point case at $x=0$. 
Since, for $0 < \varepsilon$ and $0 < x$ both sufficiently small one infers that  
\begin{align}
& (3/4)x^{\alpha-2}\sum_{j=1}^{N}\prod_{\ell=1}^{j}[\ln_{\ell}(x)]^{-2}+x^{\alpha-2}\sum_{j=1}^{N-1}\prod_{\ell=1}^{j}
[\ln_{\ell}(x)]^{-2}\sum_{m=j+1}^{N}\prod_{p=j+1}^{m}[\ln_{p}(x)]^{-1}     \no\\
& \quad \leq[(3/4)+\varepsilon] x^{\alpha-2}[\ln_{1}(x)]^{-2},      \lb{A.18}
\end{align}
condition \eqref{3.5} implies that $q(x) \geq Q_{\alpha,N,\varepsilon}(x) \geq q_{\alpha,N}(x)$ for $0 < \varepsilon$ and $0<x$ sufficiently small. Thus,   
Theorem \ref{t2.8} implies that also $\tau_{\alpha}$, $\alpha \in (-\infty,2)$, is nonoscillatory and in the limit point case at $x=0$.
\end{proof}
%%%%%%%%

Although not needed in the context of Theorem \ref{t3.1}, we note a second (necessarily nonoscillatory) linearly independent solution $\wti y_N$ of $\tau_{\alpha,N} y =0$ is a consequence of the standard reduction of order approach 
\begin{equation}
\wti y_N(x)= y_{N}(x) \int_{x}^c dt \, t^{-\alpha} y_{N}(t)^{-2}, \quad 0 < x < c \, \text{ sufficiently small.}
	\lb{3.17}
\end{equation}

%%%%%%%%
\begin{remark} \lb{r3.2}
In connection with the nonoscillatory behavior of $\tau_{\alpha}$ we now recall the power weighted and logarithmically refined Hardy inequalities in the form (see \cite{GLMP21} and the references therein)
\begin{align}
\begin{split} 
& \int_0^\rho dx \, x^{\alpha} |f'(x)|^2 \geq \f{(1 - \alpha)^2}{4} \int_0^{\rho} dx \, x^{\alpha - 2} |f(x)|^2,   \\
& \hspace*{.95cm} \alpha \in \bbR, \; \rho \in (0,\infty) \cup \{\infty\}, \; f \in C_0^{\infty}((0,\rho)),
\end{split}
	\lb{3.18}
\end{align}
and
\begin{align}
\begin{split}
& \int_0^\rho dx \, x^{\alpha} |f'(x)|^2 \geq \f{(1 - \alpha)^2}{4} \int_0^{\rho} dx \, x^{\alpha - 2} |f(x)|^2     \\
& \quad + \f{1}{4} \sum_{j=1}^N \int_0^{\rho} dx \, x^{\alpha - 2} 
\Bigg(\prod_{\ell=1}^{j} [\ln_{\ell}(x/\gamma)]^{-2}\Bigg) |f(x)|^2,  
\\
& \quad \, N \in \bbN, \; \alpha \in \bbR, \; \rho, \gamma \in (0,\infty), \; \gamma \geq e_N \rho, 
\; f \in C_0^{\infty}((0,\rho)),
\end{split}
	\lb{3.19}
\end{align}
where
\begin{equation}
e_{0} = 0, \quad e_{j+1} = e^{e_{j}}, \quad j \in \bbN_{0} = \bbN \cup \{0\}.
	\lb{3.20}
\end{equation}
Inequalities \eqref{3.18} and \eqref{3.19} imply, in particular, that
\begin{align}
\begin{split} 
\bigg(- \f{d}{dx} x^{\alpha} \f{d}{dx} - \f{(1-\alpha)^2}{4} x^{\alpha - 2}\bigg)\bigg|_{C_0^{\infty}((0,\rho))} \geq 0,& \\
\alpha \in \bbR, \; \rho \in (0,\infty) \cup \{\infty\},&
\end{split}
	\lb{3.21}
\end{align}
and 
\begin{align}
\begin{split} 
\bigg(- \f{d}{dx} x^{\alpha} \f{d}{dx} - \f{(1-\alpha)^2}{4} x^{\alpha - 2} 
- \f{1}{4} x^{\alpha - 2} \sum_{j=1}^N \prod_{\ell=1}^{j} [\ln_{\ell}(x/\gamma)]^{-2} 
\bigg)\bigg|_{C_0^{\infty}((0,\rho))} \geq 0,&   \\
N \in \bbN, \; \alpha \in \bbR, \; \rho, \gamma \in (0,\infty), \; \gamma \geq e_N \rho.& 
\end{split}
	\lb{3.22}
\end{align}
All constants displayed in \eqref{3.18}--\eqref{3.22} are sharp (cf.\ \cite{GLMP21}). 

Since $[(3/4) - (\alpha/2)] \geq - (1-\alpha)^2/4$ is equivalent to $(\alpha - 2)^2 \geq 0$, which automatically holds for 
all $\alpha \in \bbR$, the assertion that $\tau_{\alpha,0}$, $\alpha \in \bbR$, is nonoscillatory (cf.\ the paragraph following \eqref{3.9}) is of course consistent with Hardy's inequality \eqref{3.18} which implies the following (much weaker) inequality 
\begin{align}
\begin{split} 
\bigg(- \f{d}{dx} x^{\alpha} \f{d}{dx} + \bigg(\f{3}{4} - \f{\alpha}{2}\bigg) x^{\alpha - 2}\bigg)\bigg|_{C_0^{\infty}((0,\rho))} \geq 0,& \\
\alpha \in \bbR, \;\, \rho \in (0,\infty) \cup \{\infty\}.&
\end{split}
	\lb{3.23}
\end{align}
Similarly, the assertion that $\tau_{\alpha,N}$, $\alpha \in \bbR$, $N \in \bbN$, is nonoscillatory (cf.\ the paragraph following \eqref{3.14}) is consistent with the logarithmic refinements of Hardy's inequality \eqref{3.19} as 
$[(3/4) - (\alpha/2)] > - (1-\alpha)^2/4$ is equivalent to $(\alpha - 2)^2/4 > 0$, which in turn automatically holds for 
all $\alpha \in (-\infty, 2)$. The subtle difference $[(3/4) - (\alpha/2)] > - (1-\alpha)^2/4$ versus 
$[(3/4) - (\alpha/2)] \geq - (1-\alpha)^2/4$ now is crucial as the leading logarithmic terms in \eqref{3.11} are of the 
form $[\ln_{\ell}(x)]^{-1}$ as opposed to the smaller terms 
$[\ln_{\ell}(x)]^{-2}$ in \eqref{3.19} and \eqref{3.22}. In particular, \eqref{3.22} now implies 
\begin{align}
& \bigg(- \f{d}{dx} x^{\alpha} \f{d}{dx} + \bigg(\f{3}{4} - \f{\alpha}{2}\bigg) x^{\alpha - 2}  
+ \f{\alpha - 2}{2} x^{\alpha - 2} \sum_{j=1}^N \prod_{\ell=1}^{j} [\ln_{\ell}(x/\gamma)]^{-1} 
\bigg)\bigg|_{C_0^{\infty}((0,\rho))} \geq 0,    \no \\
& \hspace*{7cm} N \in \bbN, \; \alpha \in \bbR, \; \gamma \geq e_N \rho,
	\lb{3.24} 
\end{align}
for suitable $\rho, \gamma \in (0,\infty)$, since for arbitrarily small $0 < \delta \equiv (\alpha - 2)^2/4$, the expression 
$\delta x^{\alpha - 2}$ dominates the second term on the right-hand side of \eqref{3.11} in a sufficiently small neighborhood of $x=0$. 
\hfill $\diamond$
\end{remark} 
%%%%%%%%

Our second main result then reads as follows. 
%%%%%%%%
\begin{theorem} \lb{t3.3}
Suppose that $q \in L_{loc}^{1}((0,b);dx)$ is real-valued a.e.~on $(0,b)$. \\[1mm] 
$(i)$ Let $\alpha \in (-\infty,2)$ and assume that there exists $\varepsilon \in (0,1)$ $($depending on $\alpha)$ such that 
for a.e.~$0<x$ sufficiently small $($depending on $\varepsilon)$, 
\begin{equation}
q(x)\leq[(3/4)-(\alpha/2) - \varepsilon]x^{\alpha-2}.
	\lb{3.25}
\end{equation}
Then $\tau_{\alpha}$ is in the limit circle case at $x=0$. \\[1mm]
$(ii)$ Let $\alpha\in (-\infty,2)$ and assume there exist $N\in\bbN$ and $\varepsilon\in(0,1)$ $($depending on 
$\alpha$ and $N$$)$, such that 
for a.e.~$0<x$ sufficiently small $($depending on $N$ and $\varepsilon$$)$, 
\begin{align}
\begin{split} 
q(x) &\leq [(3/4)-(\alpha/2)]x^{\alpha-2} 
- (1/2) (2 - \alpha) x^{\alpha-2}\sum_{j=1}^{N}\prod_{\ell=1}^{j}[\ln_{\ell}(x)]^{-1}    \\
&\quad - (\varepsilon/2) (2 - \alpha) x^{\alpha-2}\prod_{k=1}^{N}[\ln_{k}(x)]^{-1} \equiv \hatt Q_{\alpha,N,\varepsilon}(x).      \lb{3.26}
\end{split} 
\end{align}
Then, $\tau_{\alpha}$ is in the limit circle case at $x=0$.
\end{theorem}
%%%%%%%%
\begin{proof} 
In the following $0 < x$ is assumed to be sufficiently small. \\[1mm] 
$(i)$ Abbreviating for $\alpha \in (-\infty,2)$, $\beta \in (0,\infty)$,  
\begin{align}
& q_{\alpha,0,\beta}(x)= [(3/4) - (\alpha/2) - \beta]x^{\alpha-2},    
	\lb{3.27}\\
& y_{\alpha,0,\beta,j}(x)=x^{\gamma_{\alpha,\beta,j}},      \no \\
& \gamma_{\alpha,\beta,j} = (1/2) (1 - \alpha) - (1/2) (-1)^j \begin{cases} 
\big|(2 - \alpha)^2 - 4 \beta\big|^{1/2}, & 0 < 4 \beta \leq (2-\alpha)^2, \\
i \big|(2 - \alpha)^2 - 4 \beta\big|^{1/2}, & (2-\alpha)^2 \leq 4 \beta,
\end{cases} \no \\ 
& \hspace*{3.75cm}   \alpha \in (-\infty,2), \; \beta \in (0,\infty) \backslash \big\{(2-\alpha)^2\big/4\big\}, \;  j=1,2,      \lb{3.28}\\
& y_{\alpha,0,(2 - \alpha)^2/4,1}(x) = x^{(1 - \alpha)/2}, 
\quad y_{\alpha,0,(2 - \alpha)^2/4,2}(x) = x^{(1 - \alpha)/2} \ln(1/x),        \lb{3.29} \\ 
& \hspace*{5.05cm}   \gamma_{\alpha,(2 - \alpha)^2/4,j} = (1 - \alpha)/2, \; j=1,2,     \no \\
&\tau_{\alpha,0,\beta} = - (d/dx)x^{\alpha}(d/dx)+q_{\alpha,0,\beta}(x),
	\lb{3.30}
\end{align}
one confirms that 
\begin{align}
(\tau_{\alpha,0,\beta}\,y_{\alpha,0,\beta,j})(x)=0, \quad \alpha \in (-\infty,2), \; \beta \in (0,\infty), \; j =1,2.
	\lb{3.31}
\end{align}
To verify the limit circle property of $\tau_{\alpha,0,\beta}$, one needs to guarantee that for  some $\rho \in (0,1)$, $y_{\alpha,0,\beta,j} \in L^2((0, \rho);dx)$, $j=1,2$, equivalently, 
\begin{equation}
\Re(\gamma_{\alpha,\beta,j}) > - 1/2, \quad j = 1,2.
	\lb{3.32} 
\end{equation}
Inequality \eqref{3.32} in turn is equivalent to 
\begin{equation}
\alpha \in (-\infty,2), \; \beta \in (0,\infty).      \lb{3.33}
\end{equation}
Hence, choosing $\alpha \in (2,\infty)$ and $\beta \equiv \varepsilon \in \big(0, (2 - \alpha)^2/4\big]$ yields that for $0 < \varepsilon$ sufficiently small, $\tau_{\alpha,0,\varepsilon}$ is in the limit circle case and nonoscillatory (cf.\ \eqref{3.28}) at $x=0$. 
An application of Theorem \ref{t2.9} then yields that also $\tau_{\alpha}$ is in the limit circle case at $x=0$. \\[1mm] 
$(ii)$ We recall that $q_{\al,N}$ is given by \eqref{3.11} and abbreviate for $N\in\bbN$ and $0<x$ sufficiently 
small,\footnote{\,See footnote \ref{f1}.} 
\begin{align}
& \no q_{\alpha,N,\varepsilon}(x) = q_{\al,N}(x) - 
(\eps/2) (\alpha - 2)x^{\al-2}\prod_{k=1}^{N}[\ln_{k}(x)]^{-1}
+ \big(\eps^2\big/4\big) x^{\al-2}\prod_{k=1}^{N}[\ln_{k}(x)]^{-2}
\\ & \hspace*{1.4cm} \quad  +\eps x^{\al-2}\prod_{k=1}^{N}[\ln_{k}(x)]^{-1} 
\sum_{j=1}^{N}\prod_{\ell=1}^{j}[\ln_{\ell}(x)]^{-1}, \lb{3.34} \\
&y_{N,\eps} (x)=x^{-1/2}\prod_{k=1}^{N-1}[\ln_{k}(x)]^{-1/2}[\ln_{N}(x)]^{-1/2-\varepsilon/2},
	\lb{3.35}\\
& \wti y_{N,\eps} (x)= y_{N,\eps} (x) \int_{x_0}^x dt \, t^{- \alpha} 
y_{N,\eps} (t)^{-2}, 	
	\lb{3.36}\\ 
&\tau_{\alpha,N,\varepsilon}=-(d/dx)x^{\alpha}(d/dx)+q_{\alpha,N,\varepsilon}(x). 
	\lb{3.37}
\end{align}
At this point we claim (cf.\ Lemma \ref{lA.2} for details) that 
\begin{align}
(\tau_{\alpha,N,\varepsilon}\,y_{N,\eps} )(x)=0, \quad (\tau_{\alpha,N,\varepsilon}\,\wti y_{N,\eps} )(x)=0, 
 \quad \alpha \in (-\infty,2), \; N \in \bbN.     \lb{3.38}
\end{align} 
Since for $0 < \delta_N$ sufficiently small,
\begin{equation}
y_{N,\eps}  \in L^{2}((0,\delta_N);\,dx), \quad \alpha \in (-\infty,2), \; N \in \bbN,  
	\lb{3.39}
\end{equation}
and (utilizing $\alpha \in (- \infty,2)$)
\begin{equation}
\int_0^{\delta_N} dt \, t^{1 - \alpha} \Bigg(\prod_{k=1}^{N-1} \ln_{k}(t) \Bigg) [\ln_N (t)]^{1 + \varepsilon} < \infty,
	\lb{3.40}
\end{equation}
also 
\begin{equation}
\wti y_{N,\eps}  \in L^{2}((0,\delta_N);\,dx), \quad \alpha \in (-\infty,2), \; N \in \bbN.   
	\lb{3.41}
\end{equation}
Thus, $\tau_{\alpha,N,\varepsilon}$ is in the limit circle case and nonoscillatory (cf.\ \eqref{3.35}) at $x=0$.  
Moreover, since for $0 < \varepsilon$ and $0<x$ sufficiently small,
\begin{align}
\begin{split} 
q_{\alpha,N,\varepsilon}(x) &\geq [(3/4)-(\alpha/2)]x^{\alpha-2} 
- (1/2) (2 - \alpha) x^{\alpha-2}\sum_{j=1}^{N}\prod_{\ell=1}^{j}[\ln_{\ell}(x)]^{-1}    \\
&\quad - (\varepsilon/2) (2 - \alpha) x^{\alpha-2}\prod_{k=1}^{N}[\ln_{k}(x)]^{-1}, 
\end{split} 
\end{align}
condition \eqref{3.26} implies that $q(x) \leq \hatt Q_{\alpha,N,\eps}(x) \leq q_{\alpha,N,\varepsilon}(x)$ for 
$0 < \varepsilon$ and $0<x$ sufficiently small. Thus,   
Theorem \ref{t2.9} implies that also $\tau_{\alpha}$, $\alpha \in (-\infty,2)$, is in the limit circle case at $x=0$. 
\end{proof}
%%%%%%%%

%%%%%%%%
\begin{remark} \lb{r3.4}
Thus far we focusd on endpoint classifications at $x=0$. It is of course possible to address the limit point 
case at $x=\infty$, however, this is typically based on a markedly different technique and we provide an example next: For this purpose we abbreviate iterated logarithms for $0 < x$ sufficiently large by  
\begin{equation}
{\rm Ln}_1(x) = \ln(x), \quad {\rm Ln}_{j+1}(x) = \ln({\rm Ln}_{j}(x)), \quad j \in \bbN.
\end{equation}
Next, let $a \in \bbR$, $\alpha \in (-\infty,2]$, and 
consider $\tau_{\alpha}$ in \eqref{3.2} on the interval $[a,\infty)$. Suppose there exists $C \in (0,\infty)$ such that for 
$R \in (a,\infty)$ sufficiently large (rendering ${\rm Ln}_{\ell}(\dott)$, $1 \leq \ell \leq N$, in \eqref{3.43} strictly positive), and some $N \in \bbN$, 
\begin{equation}
q(x) \geq - C x^{2-\alpha} \prod_{k=1}^N [{\rm Ln}_{k}(x)]^2 \, \text{ for a.e.~$x \in [R,\infty)$.}     \lb{3.43} 
\end{equation}
Then $\tau_{\alpha}$, $\alpha \in (-\infty,2]$, is in the limit point case at $x = \infty$.

For the proof it suffices to choose 
\begin{equation}
p(x) = x^{\alpha}, \quad  
M(x) = x^{2 - \alpha} \prod_{k=1}^N [{\rm Ln}_{k}(x)]^2, \quad x \in [R,\infty),
\end{equation} 
and refer to \cite[Theorem~16, p.~1406]{DS88}. (For a three-coefficient analog of the latter, see 
\cite[Theorem~7.4.3, p.~148--149]{Ze05}.)

For additional results regarding the absence of $L^2$-solutions at $x = \infty$, implying the limit point case at $x=\infty$, we also refer to \cite{Se49}, \cite{Se52}, \cite{Se54}. 
\hfill $\diamond$
\end{remark}
%%%%%%%%%

%%%%%%%%%%
%%%%%%%%%%
\section{An Elementary Multi-Dimensional Application} \lb{s4}
%%%%%%%%%%
%%%%%%%%%%

In this section we briefly sketch an elementary multi-dimensional application of Theorem \ref{t3.1} in connection with the partial differential expression $- \Div (p(|\dott|) \nabla) + q(|\dott|)$. 

In $n$-dimensional spherical coordinates, the differential expression $- \Div (p(|\dott|) \nabla)$ on the $n$-dimensional ball $B_n(0;R) \subset \bbR^n$, $n \in \bbN$, $n \geq 2$, $R \in (0,\infty)$, assuming
\begin{equation}
1/p \in L^1((\varepsilon,R); dr), \, 0 < p \in AC([\varepsilon,R]) \, \text{ for all $\varepsilon > 0$},    \lb{4.1} 
\end{equation}
 takes the form
\begin{align}
- \Div p(|x|) \nabla = - r^{1-n} \f{\partial}{\partial r} \bigg(r^{n-1} p(r) \f{\partial}{\partial r}\bigg)  - \f{p(r)}{r^{2}} \Delta_{\bbS^{n-1}}, \quad x \in B_n(0;R) \backslash \{0\},
       \lb{4.2}
\end{align}
where $-  \Delta_{\bbS^{n-1}}$ denotes the Laplace--Beltrami operator associated with the $(n-1)$-dimensional unit sphere $\bbS^{n-1}$ in $\bbR^n$. When acting in $L^{2}(B_n(0;R))$, which in spherical coordinates can be written as 
$L^{2}(B_n(0;R); d^n x) \simeq L^{2}((0,R);r^{n-1}dr)\otimes L^{2}(\bbS^{n-1})$, \eqref{4.2}  becomes 
\begin{align}
- \Div p(|x|) \nabla = \bigg[- \f{d}{dr} p(r) \f{d}{dr} - \f{(n-1) p(r)}{r}\f{d}{dr}\bigg] \otimes I_{L^{2}(\bbS^{n-1})} - \f{p(r)}{r^{2}}\otimes  \Delta_{\bbS^{n-1}}
      \lb{4.3}
\end{align}
(with $I_{\cX}$ denoting the identity operator on $\cX$).  
The Laplace--Beltrami operator $-  \Delta_{\bbS^{n-1}}$ in $L^{2}(\bbS^{n-1})$, with domain $\dom(- \Delta_{\bbS^{n-1}}) = H^2\big(\bbS^{n-1}\big)$ (cf., e.g., \cite{BLP19}),  is known to be essentially self-adjoint and nonnegative on $C_0^{\infty}(\bbS^{n-1})$ (cf.\ \cite[Theorem~5.2.3]{Da89}). Recalling the  treatment in \cite[p.~160--161]{RS75}, one decomposes the space $L^{2}(\bbS^{n-1})$ into an infinite orthogonal sum, yielding
\begin{align}
\begin{split} 
L^{2}(B_n(0;R); d^n x) &\simeq L^{2}((0,R);r^{n-1}dr)\otimes L^{2}(\bbS^{n-1})    \\
& = \bigoplus\limits_{\ell=0}^{\infty}L^{2}((0,R);r^{n-1}dr)\otimes\cY_{\ell}^{n},   \lb{4.4}
\end{split} 
\end{align}
where $\cY_{\ell}^{n}$ is the eigenspace of $- \Delta_{\bbS^{n-1}}$ corresponding to the eigenvalue $\ell(\ell+n-2)$, $\ell \in \bbN_0$, as 
\begin{equation}
\sigma (- \Delta_{\bbS^{n-1}}) = \{\ell(\ell+n-2)\}_{\ell \in \bbN_0}.     \lb{4.5} 
\end{equation}
In particular, this results in 
\begin{align}
- \Div p(|x|) \nabla = \bigoplus\limits_{\ell=0}^{\infty} \bigg[- \f{d}{dr} p(r) \f{d}{dr} - \f{(n-1) p(r)}{r}\f{d}{dr} 
+ \f{\ell(\ell+n-2) p(r)}{r^{2}}\bigg] 
\otimes I_{\cY_{\ell}^{n}},       \lb{4.6}
\end{align}
in the space \eqref{4.4}.

To simplify matters, replacing the measure $r^{n-1}dr$ by $dr$ and simultaneously removing the term 
$(n-1) p(r) r^{-1} (d/dr)$, one introduces the unitary operator
\begin{align}
U_{n}=
\begin{cases}
L^{2}((0,R);r^{n-1}dr)\rightarrow L^{2}((0,R);dr),  \\[1mm] 
f(r)\mapsto r^{(n-1)/2}f(r),
\end{cases}      \lb{4.7}
\end{align}
under which \eqref{4.6} becomes
\begin{align}
- \Div p(|x|) \nabla &= \bigoplus\limits_{\ell=0}^{\infty} U_{n}^{-1}\bigg[-\f{d}{dr} p(r) \f{d}{dr}  + (n-1) \f{p'(r)}{2r}  
 \lb{4.8} \\
& \hspace*{1.9cm} + \{[(n-1)(n-3)/4] + \ell(\ell+n-2)\} \f{p(r)}{r^2}\bigg] U_{n} \otimes I_{\cY_{\ell}^{n}},     \no 
\end{align}
still acting in the space \eqref{4.4}. Thus, specializing to the case
\begin{equation}
p(r) = r^{\alpha}, \quad \alpha \in \bbR, \; r \in (0,R],    \lb{4.9} 
\end{equation} 
the self-adjoint Friedrichs $L^2$-realization, $H^{(0)}_{\alpha,F}$, of 
$- \Div |\dott|^{\alpha} \nabla$ in the space \eqref{4.4} then is of the form
\begin{equation}
H^{(0)}_{\alpha,F} = \bigoplus\limits_{\ell=0}^{\infty} U_{n}^{-1} h^{(0)}_{n,\ell,\alpha,F} \,U_{n} \otimes I_{\cY_{\ell}^{n}},
	\lb{4.10}
\end{equation}
where $h^{(0)}_{n,\ell,\alpha,F}$, $\ell \in \bbN_0$, represents the Friedrichs extension of the preminimal operator, 
$\dot h^{(0)}_{n,\ell,\alpha}$ in $L^2((0,R); dr)$, associated with the differential expression 
\begin{align}
\begin{split} 
\tau^{(0)}_{n,\ell,\alpha} = -\f{d}{dr} r^{\alpha} \f{d}{dr} + \f{[(n-1)(n-3 + 2 \alpha)/4] + \ell(\ell+n-2)}{r^{2-\alpha}},& \\ 
\alpha \in \bbR, \; n \in \bbN, \, n \geq 2, \; \ell \in \bbN_0, \; r \in (0,R],&     \lb{4.11} 
\end{split} 
\end{align} 
that is, 
\begin{align}
\dot h^{(0)}_{n,\ell,\alpha} = \tau^{(0)}_{n,\ell,\alpha}\big|_{C_0^{\infty}((0,R))}, \quad 
\alpha \in \bbR, \; n \in \bbN, \, n \geq 2, \; \ell \in \bbN_0,&       \lb{4.12}
\end{align}
in $L^2((0,R); dr)$. 

To explicitly describe $h^{(0)}_{n,\ell,\alpha,F}$ we next recall some results from \cite{GLN20} and \cite{GNS21}. For this purpose we introduce the differential expression 
\begin{align}
\begin{split}
\tau_{\b,\g} = \left[-\frac{d}{dx}x^\b\frac{d}{dx} +\frac{(2-\b)^2\g^2-(1-\b)^2}{4}x^{\b-2}\right],\\
\b \in \bbR,\; \g \in [0,\infty), \; x\in(0,R],     \lb{4.13} 
\end{split}
\end{align}
and recall that solutions to $\tau_{\b,\g} y(z, \dott) = z y(z,\dott)$  are given by (cf.\ \cite[No.~2.162, p.~440]{Ka61})
\begin{align}
y_{1,\b,\g}(z,x)&=x^{(1-\b)/2} J_{\gamma}\big(2z^{1/2} x^{(2-\b)/2}/(2-\b)\big),\quad \g \in [0,\infty),\\
y_{2,\b,\g}(z,x)&=\begin{cases}
x^{(1-\b)/2} J_{-\gamma}\big(2z^{1/2} x^{(2-\b)/2}/(2-\b)\big), & \g\notin\bbN_0,\\
x^{(1-\b)/2} Y_{\g}\big(2z^{1/2} x^{(2-\b)/2}/(2-\b)\big), & \g\in\bbN_0,
\end{cases} \; \g \in [0,\infty),   \\
& \hspace*{8.15cm} x \in (0,R].   \no 
\end{align}
where $J_{\nu}(\dott), Y_{\nu}(\dott)$ are the standard Bessel functions of order $\nu \in \bbR$ 
(cf.\ \cite[Ch.~9]{AS72}). Solutions for $z=0$ are particularly simple and we note that (non-normalized) principal and nonprincipal solutions $u_{0,\b,\gamma}(0, \dott)$ and $\hatt u_{0,\b,\gamma}(0, \dott)$ of $\tau_{\beta,\gamma} u = 0$ at $x=0$ are of the form 
\begin{align}
\begin{split} 
u_{0,\b,\gamma}(0, x) &= x^{[1-\b+(2-\b)\g]/2}, \quad \gamma \in [0,\infty),   \\
\hatt u_{0,\b,\gamma}(0, x) &= \begin{cases} x^{[1-\b-(2-\b)\g]/2}, & \gamma \in (0,\infty),     \lb{4.16} \\
x^{(1-\b)/2} \ln(1/x), & \gamma =0,  \end{cases}\\
&\hspace*{1.9cm} \beta \in \bbR,\; x \in (0,1),      \\
\hatt u_{0,2,\gamma}(0, x) &= x^{-1/2} \ln(1/x), \quad \gamma \in [0,\infty), \; x \in (0,1). 
\end{split} 
\end{align}
In particular, $\tau_{\b,\g}$, $\beta \in \bbR$, $\gamma \in [0,\infty)$, is nonoscillatory at $x=0$ and $x=R$, regular at $x=R$, and the following limit point/limit circle classification holds, 
\begin{equation}
\begin{cases}
\text{$\tau_{\b,\g}$ is in the limit point case at $x=0$ if $\beta \in [2,\infty)$, $\gamma \in [0,\infty)$} \\
\quad \text{and if $\beta \in (-\infty,2)$, $\gamma \in [1,\infty)$,} \\
\text{$\tau_{\b,\g}$ is in the limit circle case at $x=0$ if $\beta \in (-\infty,2)$, $\gamma \in [0,1)$,} \\
\text{$\tau_{\b,\g}$ is in the limit circle case at $x=R$ if $\beta \in \bbR$, $\gamma \in [0,\infty)$.}  \lb{4.17} 
\end{cases} 
\end{equation}
The preminimal, $ \dot T_{\beta,\gamma}$, and maximal $T_{\beta,\gamma, max}$, $L^2((0,R]; dx)$-realizations associated with $\tau_{\b,\g}$, $\beta \in \bbR$, $\gamma \in [0,\infty)$, are then given by 
\begin{align}
& \dot T_{\beta,\gamma} = \tau_{\b,\g}\big|_{C_0^{\infty}((0,R))},    \\
& (T_{\beta,\gamma, max} f)(x) = (\tau_{\beta,\gamma} f)(x) \, \text{ for a.e.~$x \in (0,R]$,}    \no \\
& \, f \in \dom(T_{\beta,\gamma, max}) = \big\{g \in L^2((0,R); dx) \, \big | \, 
g, g' \in AC_{loc}((\varepsilon,R]) \, \text{for all $0 < \varepsilon < R$};      \no \\  
& \hspace*{7.5cm} \tau_{\beta,\gamma} g \in L^2((0,R); dx) \big\},    
\end{align}

According to \cite{GLN20}, the generalized boundary values for $g \in \dom(T_{\beta, \gamma,max})$ at $x=0$ in the limit circle case at $x=0$ (i.e., if $\beta \in (-\infty,2)$, $\gamma \in [0,1)$) are of the form
\begin{align}
\begin{split} 
\wti g(0) &= \begin{cases} \lim_{x \downarrow 0} g(x)\big/\big[x^{[1-\b-(2-\b)\g]/2}\big], & 
\gamma \in (0,1), \\[1mm]
\lim_{x \downarrow 0} g(x)\big/\big[x^{(1-\b)/2} \ln(1/x)\big], & \gamma =0, 
\end{cases} 
\end{split} \\
\wti g^{\, \prime} (0) 
&= \begin{cases} \lim_{x \downarrow 0} \big[g(x) - \wti g(0) x^{[1-\b-(2-\b)\g]/2}\big] 
\big/\big[x^{[1-\b+(2-\b)\g]/2}\big], 
& \hspace{-.2cm}   \gamma \in (0,1),     \\[1mm]
\lim_{x \downarrow 0} \big[g(x) - \wti g(0) x^{(1-\b)/2} \ln(1/x)\big] \big/\big[x^{(1-\b)/2}\big], 
& \hspace{-.2cm} \gamma =0.
\end{cases}
\end{align}
Since $\tau_{\b,\g}$ is regular at $x=R$, the standard boundary values for $g \in \dom(T_{\beta, \gamma,max})$ at $x=R$ are of the standard form $g(R), g'(R)$.

The closure of $\dot T_{\beta,\gamma}$ in $L^2((0,R); dx)$, that is, the minimal operator, $T_{\beta,\gamma, min}$, associated with 
$\tau_{\b,\g}$, is then given by
\begin{align} 
& (T_{\beta,\gamma, min} f)(x) = (\tau_{\beta,\gamma} f)(x) \, \text{ for a.e.~$x \in (0,R]$, $\beta \in \bbR$, 
$\gamma \in [0,\infty)$,}   \no \\
& \, f \in \dom(T_{\beta,\gamma, min}) = \big\{g \in \dom(T_{\beta,\gamma, max}) \, \big | \, 
\wti g (0) = \wti g^{\, \prime} (0) = 0, \, g(R) = g'(R) = 0 \big\},    \no \\
& \hspace*{7.3cm} \beta \in (-\infty,2), \; \gamma \in [0,1),      \\ 
& \, f \in \dom(T_{\beta,\gamma, min}) = \big\{g \in \dom(T_{\beta,\gamma, max}) \, \big | \, 
g(R) = g'(R) = 0 \big\},     \\
& \hspace*{1.52cm} \beta \in (-\infty,2), \; \gamma \in [1,\infty), \, \text{ or, } \, \beta \in [2,\infty), \; \gamma \in [0,\infty),  
\no 
\end{align}
and the Friedrichs extension, $T_{\beta, \gamma,F}$, of $T_{\beta,\gamma, min}$ (and $\dot T_{\beta,\gamma}$) is characterized by (cf.\ \cite{Ka78}, \cite{NZ92}, \cite{Ro85}), 
\begin{align}
& (T_{\beta,\gamma, F} f)(x) = (\tau_{\beta,\gamma} f)(x) \, \text{ for a.e.~$x \in (0,R]$, $\beta \in \bbR$, 
$\gamma \in [0,\infty)$,}   \no \\
& \, f \in \dom(T_{\beta,\gamma, F}) = \big\{g \in \dom(T_{\beta,\gamma, max}) \, \big | \, 
\wti g (0) =0, \, g(R) = 0\big\},   \lb{4.24}  \\
& \hspace*{5.8cm} \beta \in (-\infty,2), \; \gamma \in [0,1),     \no \\ 
& \, f \in \dom(T_{\beta,\gamma, F}) = \big\{g \in \dom(T_{\beta,\gamma, max}) \, \big | \, g(R) = 0\big\},   \lb{4.25}  \\
& \hspace*{-.7mm} \beta \in (-\infty,2), \; \gamma \in [1,\infty), \, \text{ or, } \, \beta \in [2,\infty), \; \gamma \in [0,\infty). 
\no
\end{align}

Returning to $\tau^{(0)}_{n,\ell,\alpha}$, and hence comparing 
\begin{equation}
(n-1)(n-3 + 2 \alpha) + 4 \ell (\ell + n -2) \, \text{ with } \, (2-\alpha)^2 \gamma^2 -(1 - \alpha)^2 \, \text{ for } \, \alpha \in \bbR \backslash \{0\},
\end{equation}
and treating the case $\alpha = 2$ separately, 
an application of \eqref{4.13}--\eqref{4.25} then yields the following facts for its limit point/limit circle classification, for the maximal operator $h^{(0)}_{n,\ell,\alpha,max}$ associated with 
$\tau^{(0)}_{n,\ell,\alpha}$, and the Friedrichs extension $h^{(0)}_{n,\ell,\alpha,F}$ of $\dot h^{(0)}_{n,\ell,\alpha}$ in $L^2((0,R); dx)$. First, upon identifying
\begin{align}
& \beta = \alpha \in \bbR \backslash \{2\}, \quad \gamma= \gamma_{\alpha}  
= \big[(2 - \alpha - n)^2 + 4 \ell (\ell + n - 2)\big]^{1/2} \big/ |2 - \alpha| \in [0,\infty), 
\end{align}
more precisely, 
\begin{align}
\begin{split}
& \alpha > 2, \quad \gamma_{\alpha} \in [1,\infty),    \\
& \alpha < 2, \quad \gamma_{\alpha} \in [0,\infty), 
\end{split}
\end{align}
$\tau^{(0)}_{n,\ell,\alpha}$, $\alpha \in \bbR$, is nonoscillatory at $r=0$ and $r=R$, and regular at $r=R$. In addition,   
\begin{equation}
\begin{cases}
\text{$\tau^{(0)}_{n,\ell,\alpha}$ is in the limit point case at $r=0$ if $\alpha \in (2,\infty)$, 
$\gamma_{\alpha} \in [0,\infty)$,} \\
\quad \text{if $\alpha \in (-\infty,2)$, $\gamma_{\alpha} \in [1,\infty)$, and if $\alpha = 2$,} \\
\text{$\tau^{(0)}_{n,\ell,\alpha}$ is in the limit circle case at $r=0$ if $\alpha \in (-\infty,2)$, 
$\gamma_{\alpha} \in [0,1)$,} \\
\text{$\tau^{(0)}_{n,\ell,\alpha}$ is in the limit circle case at $r=R$ for all $\alpha \in \bbR$,}  
\end{cases}       \lb{4.29} 
\end{equation}
and hence 
\begin{equation}
\begin{cases}
\text{$\tau^{(0)}_{n,\ell,\alpha}$ is in the limit point case at $r=0$ if and only if} \\
\quad \text{$\alpha \in [2 - (n/2) - (2/n)\ell(\ell + n - 2),\infty)$,}    \\
\text{$\tau^{(0)}_{n,\ell,\alpha}$ is in the limit circle case at $r=0$ if and only if} \\ 
\quad \text{$\alpha \in (-\infty, 2 - (n/2) - (2/n)\ell(\ell + n - 2))$,} \\
\text{$\tau^{(0)}_{n,\ell,\alpha}$ is in the limit circle case at $r=R$ for all $\alpha \in \bbR$.}  \lb{4.30} 
\end{cases}
\end{equation}
Moreover, the underlying maximal operator is of the form 
\begin{align}
& \big(h^{(0)}_{n,\ell,\alpha,max} f\big)(r) = \big(\tau^{(0)}_{n,\ell,\alpha} f\big)(r)  \, \text{ for a.e.~$ r \in (0,R]$, $\alpha \in \bbR$,}    \no \\
& \, f \in \dom(h^{(0)}_{n,\ell,\alpha,max}) = \big\{g \in L^2((0,R);dr) \, \big| \, g, g' \in AC([\varepsilon,R]) 
\, \text{for all $\varepsilon \in (0,R)$};     \no \\
& \hspace*{6.5cm} \big(\tau^{(0)}_{n,\ell,\alpha} f\big)(r) \in L^2((0,R);dr)\big\},  \lb{4.31}
\end{align}
and the corresponding Friedrichs extension of $\dot h^{(0)}_{n,\ell,\alpha}$ is given by 
\begin{align}
& \big(h^{(0)}_{n,\ell,\alpha,F} f\big)(r) = \big(\tau^{(0)}_{n,\ell,\alpha} f\big)(r)  \, \text{ for a.e.~$ r \in (0,R]$, $\alpha \in \bbR$,}    \no \\
& \, f \in \dom(h^{(0)}_{n,\ell,\alpha,F}) = \big\{g \in \dom(h^{(0)}_{n,\ell,\alpha,max}) \, \big| \, 
\wti g (0) = 0, \, g(R) =0 \big\},  \lb{4.32} \\
& \hspace*{3.6cm} \alpha \in (- \infty, 2 - (n/2) - (2/n) \ell (\ell + n - 2)),    \no \\
& \, f \in \dom(h^{(0)}_{n,\ell,\alpha,F}) = \big\{g \in \dom(h^{(0)}_{n,\ell,\alpha,max}) \, \big| \, 
g(R) =0 \big\},  \lb{4.33} \\
& \hspace*{2.43cm} \alpha \in [2 - (n/2) - (2/n) \ell (\ell + n - 2), \infty).   \no
\end{align}
Here the boundary value $\wti g(0)$ associated with $g \in \dom\big(h^{(0)}_{n,\ell,\alpha,max}\big)$, 
$n \in \bbN$, $n \geq 2$, $\ell \in \bbN_0$, $\alpha \in (- \infty,2)$, is now given by 
\begin{align}
\wti g(0) &= \begin{cases} \lim_{x \downarrow 0} g(x)\big/\big[x^{[1-\alpha-(2-\alpha)\gamma_{\alpha}]/2}\big], & 
\gamma_{\alpha} \in (0,1), \\[1mm]
\lim_{x \downarrow 0} g(x)\big/\big[x^{(1-\alpha)/2} \ln(1/x)\big], & \gamma_{\alpha} =0, 
\end{cases}     \no \\ 
& \hspace*{5cm} \alpha \in (-\infty,2),     \no \\
&= \begin{cases} \lim_{x \downarrow 0} g(x)\big/\big[x^{\{1-\alpha-[(2-\alpha-n)^2 + 4\ell(\ell+n-2)]^{1/2}\}/2}\big],  \\
\quad \alpha \in (-\infty, 2 - (n/2) - (2/n)\ell(\ell+n-2)), \; \ell \in \bbN_0, \\[1mm]
\lim_{x \downarrow 0} g(x)\big/\big[x^{(1-\alpha)/2} \ln(1/x)\big], \;\; \alpha = 2-n, \; \ell=0. 
\end{cases} 
\end{align}

Without going into details, we note that utilizing the transformation \eqref{4.7}, and invoking results of Kalf \cite{Ka72}, the operator $H^{(0)}_{\alpha,F}$ is of the following form,
\begin{align}
& \big(H^{(0)}_{\alpha,F} \psi\big)(x) = - (\Div |x|^{\alpha} \nabla \psi)(x), \quad x \in B_n(0;R) \backslash \{0\},   \no \\
& \, \psi \in \dom\big(H^{(0)}_{\alpha}\big) = \big\{\phi \in \dom\big(H^{(0)}_{\alpha,max}\big) \, \big| \, | \dott |^{\alpha} 
(\nabla \phi) \in L^2(B_n(0;R); d^n x);    \\
& \hspace*{3.05cm} \lim_{r\uparrow R} \int_{\bbS^{n-1}} d^{n-1} \omega \, |\phi(r \omega)|^2 = 0,    \no \\ 
& \hspace*{3.1cm} \text{and if and only if $n < 2 - \alpha$,} \, \lim_{r\downarrow 0} \int_{\bbS^{n-1}} d^{n-1} \omega \, |\phi(r \omega)|^2 = 0 \big\}     \no 
\end{align}
(with $d^{n-1} \omega$ the surface measure on $\bbS^{n-1}$), where
\begin{align}
& \big(H^{(0)}_{\alpha, max} \psi\big)(x) = - (\Div |x|^{\alpha} \nabla \psi)(x), \quad x \in B_n(0;R) \backslash \{0\},  \no \\
& \, \psi \in \dom\big(H^{(0)}_{\alpha, max}\big) = \big\{\phi \in L^2(B_n(0;R); d^n x) \, \big| \, 
\phi \in H^2_{\loc}(B_n(0;R)\backslash\{0\});        \\
& \hspace*{5.8cm} \Div | \dott |^{\alpha} \nabla \phi \in L^2(B_n(0;R); d^n x)\big\}.    \no 
\end{align} 
We remark that the boundary condition at $x=R$ (and at $x=0$ if and only if $n < 2 - \alpha$) has to be imposed on a distinguished representative of $\phi$ for which the restriction to the $(n-1)$-dimensional sphere $\bbS^{n-1}$ exists as a square integrable function (see the discussion in \cite[Remark~3]{Ka72}).

We conclude these considerations by adding an additional potential term $q$ in accordance with inequalities \eqref{3.4}, \eqref{3.5}. For this purpose we assume that for all $\eta \in (0,R)$, 
\begin{equation}
q \in L^1((\eta, R); dr) \, \text{ is real-valued a.e.~on $(0,R)$,}     \lb{4.37} 
\end{equation}
and introduce for $n \in \bbN$, $n \geq 2$, $\ell \in \bbN_0$, 
\begin{equation}
\tau_{n,\ell,\alpha} = \tau^{(0)}_{n,\ell,\alpha} + q(r) \, \text{ for a.e.~$r \in (0,R)$, $\alpha \in \bbR$,}
\end{equation}
and the following $L^2((0,R); dr)$-realization of $\tau_{n,\ell,\alpha}$,
\begin{align}
& \big(h_{n,\ell,\alpha} f\big)(r) = \big(\tau_{n,\ell,\alpha} f\big)(r) \, \text{ for a.e.~$ r \in (0,R]$, $\alpha \in \bbR$,}    
\no \\
& \, f \in \dom(h_{n,\ell,\alpha}) = \big\{g \in L^2((0,R);dr) \, \big| \, g, g' \in AC([\varepsilon,R]) 
\, \text{for all $\varepsilon \in (0,R)$};      \\
& \hspace*{5.5cm} g(R) = 0; \, \big(\tau^{(0)}_{n,\ell,\alpha} f\big)(r) \in L^2((0,R);dr)\big\}.  \no
\end{align}

%%%%%%
\begin{theorem} \lb{t4.1} 
Assume $n \in \bbN$, $n\geq 2$, and \eqref{4.37}. \\[1mm] 
$(i)$ If $\alpha \in \bbR$, $\ell \in \bbN_0$, and for a.e.~$0<r$ sufficiently small, 
\begin{equation}
q(r) \geq - \{[n(n-4+2 \alpha)/4] + \ell(\ell+n-2)\} r^{\alpha-2},     \lb{4.40} 
\end{equation}
then $\tau_{n,\ell,\alpha}$ is nonoscillatory and in the limit point case at $r=0$. \\[1mm]
$(ii)$ If\,\footnote{Again, only $\alpha \in (- \infty,2)$ can improve on item $(i)$.} $\alpha\in (-\infty,2)$, 
$\ell \in \bbN_0$, and there exist $N\in\bbN$ and $\varepsilon>0$, such 
that for a.e.~$0<r$ sufficiently small $($depending on $N$ and $\varepsilon$$)$,
\begin{align}
\begin{split} 
q(r) & \geq - \{[n(n-4+2 \alpha)/4] + \ell(\ell+n-2)\} r^{\alpha-2}    \lb{4.41} \\
& \quad - (1/2) (2 - \alpha) r^{\alpha-2} 
\sum_{j=1}^{N}\prod_{\ell=1}^{j}[\ln_{\ell}(r)]^{-1} + [(3/4) + \varepsilon] r^{\alpha-2}[\ln_{1}(x)]^{-2}, 
\end{split} 
\end{align}
then $\tau_{n,\ell,\alpha}$ is nonoscillatory and in the limit point case at $r=0$. \\[1mm] 
$(iii)$ Assuming inequality \eqref{4.40} $($for $\alpha \in \bbR$$)$ or \eqref{4.41} $($for $\alpha \in (-\infty,2)$$)$ holds for $\ell=0$, then the operator
\begin{align}
H_{\alpha}=\bigoplus\limits_{\ell\in\bbN_{0}}U_{n}^{-1} h_{n,\ell,\alpha}U_{n}       \lb{4.42}
\end{align}
is self-adjoint in $L^{2}(B_{n}(0;R); d^n x)$.   
\end{theorem}
%%%%%%
\begin{proof}
Items $(i)$ and $(ii)$ are an immediate consequence of Theorem \ref{t3.1} since, for instance, inequality \eqref{4.40} in the context $N=0$ is equivalent to
\begin{equation}
\{[(n-1)(n-3+2\alpha)/4] + \ell (\ell+n-2)\} r^{\alpha-2} + q(r) \geq [(3/4) - (\alpha/2)] r^{\alpha-2}
\end{equation}
for $0 < r$ sufficiently small (cf.\ \eqref{3.4}), and analogously in the context of \eqref{4.41} (cf.\ \eqref{3.5}) for $N \in \bbN$. Item $(iii)$ holds since 
$h_{n,0,\alpha}$ self-adjoint in $L^2((0,R); dr)$ implies that $h_{n,\ell,\alpha}$ is self-adjoint in $L^2((0,R); dr)$ for all 
$\ell \in \bbN_0$. 
\end{proof}
%%%%%%

In the case $N=0$, self-adjointness of $H_{\alpha}$ is familiar from multi-dimensional results in Kalf and Walter 
\cite{KW72} (with strict inequality in the analog of \eqref{4.40} for $\ell=0$); in this context see also \cite{KSWW75}, \cite{Sc72} for $\alpha =0$.

%%%%%%%%%%%%%%%%%%%%%%%%%%%%%%%%%%%%%
\appendix
%%%%%%%%%%%%%%% Appendix A %%%%%%%%%%%%%%%%
\section{More Details in Connection with Theorems \ref{t3.1} and \ref{t3.3}} \lb{sA}
\renewcommand{\theequation}{A.\arabic{equation}}
\renewcommand{\thetheorem}{A.\arabic{theorem}}
\setcounter{theorem}{0} \setcounter{equation}{0}
%%%%%%%%%%%%%%%%%%%%%%%%%%%%%%%%%%%%%
%%%%%%%%%%%%%%%%%%%%%%%%%%%%%%%%%%%%%

In this appendix we elaborate on the proofs of Theorem \ref{t3.1} and \ref{t3.3} by sketching the proofs of the assertions in \eqref{3.14} and \eqref{3.38}.  

We begin with the limit point case discussed in Theorem \ref{t3.1}. 

%%%%%%
\begin{lemma}	\lb{lA.1}
Let the assumptions of Theorem \ref{t3.1} be satisfied, and $q_{\alpha,N}(x)$, $\tau_{\alpha,N}$, and $y_N(x)$ be as in \eqref{3.11}--\eqref{3.13}.  Then, for all $N\in\bbN$,
\begin{align}
(\tau_{\alpha,N}\,y_N)(x)=0.      \lb{A.1}
\end{align}
\end{lemma}
%%%%%%
\begin{proof}
One observes,\footnote{\,See footnote \ref{f1}.}  
\begin{align}
& \big(\ln_{N}(x)\big)'=-x^{-1}\prod_{k=1}^{N-1}[\ln_{k}(x)]^{-1},    \lb{A.2} \\
& \left([\ln_{N}(x)]^{-1/2}\right)'=\f{1}{2}x^{-1} \Bigg(\prod_{k=1}^{N-1}[\ln_{k}(x)]^{-1}\Bigg)[\ln_{N}(x)]^{-3/2},   \lb{A.3} \\
& \left(\prod_{\ell=1}^{N}[\ln_{\ell}(x)]^{-1/2}\right)'=\f{1}{2}x^{-1}\prod_{k=1}^{N}[\ln_{k}(x)]^{-1/2}
\sum_{j=1}^{N}\prod_{\ell=1}^{j}[\ln_{\ell}(x)]^{-1},   \lb{A.4} \\ 
& \Bigg(\prod_{k=1}^{N}[\ln_{k}(x)]^{-1}\Bigg)' 
= x^{-1}\prod_{k=1}^{N}[\ln_{k}(x)]^{-1}\sum_{j=1}^{N}\prod_{\ell=1}^{j}[\ln_{\ell}(x)]^{-1},     \lb{A.5}
\end{align}
and hence verifies 
\begin{align}
& \f{1}{y_{N}(x)}\big(x^{\alpha}y_{N}'(x)\big)'\no\\
&\quad=\f{1}{y_{N}(x)}\f{d}{dx}\Bigg[-\f{1}{2}x^{\alpha-3/2}\prod_{k=1}^{N}[\ln_{k}(x)]^{-1/2}     \no\\
&\qquad+\f{1}{2}x^{\alpha-3/2}\prod_{k=1}^{N}[\ln_{k}(x)]^{-1/2}\sum_{j=1}^{N}\prod_{\ell=1}^{j}[\ln_{\ell}(x)]^{-1}\Bigg] 
\no\\
&\quad=\f{1}{y_{N}(x)}\f{d}{dx}\left[\f{1}{2}x^{\alpha-3/2}\prod_{k=1}^{N}[\ln_{k}(x)]^{-1/2}\Bigg(-1+\sum_{j=1}^{N} 
\prod_{\ell=1}^{j}[\ln_{\ell}(x)]^{-1}\Bigg)\right]\no\\
&\quad=\f{1}{y_{N}(x)}\Bigg[\f{1}{2}\left(\alpha-\f{3}{2}\right)x^{\alpha-5/2}\prod_{k=1}^{N}[\ln_{k}(x)]^{-1/2}\Bigg(-1+ \sum_{j=1}^{N}\prod_{\ell=1}^{j}[\ln_{\ell}(x)]^{-1}\Bigg)\no\\
&\qquad+\f{1}{4}x^{\alpha-5/2}\prod_{k=1}^{N}[\ln_{k}(x)]^{-1/2}\sum_{j=1}^{N}\prod_{\ell=1}^{j}[\ln_{\ell}(x)]^{-1}\Bigg(-1+ \sum_{m=1}^{N}\prod_{p=1}^{m}[\ln_{p}(x)]^{-1}\Bigg)\no\\
&\qquad +\f{1}{2}x^{\alpha-5/2}\prod_{k=1}^{N}[\ln_{k}(x)]^{-1/2}\sum_{j=1}^{N}\prod_{\ell=1}^{j}[\ln_{\ell}
(x)]^{-1}\sum_{m=1}^{j}\prod_{p=1}^{m}[\ln_{p}(x)]^{-1}\Bigg]      \no 
\end{align}
\begin{align}
& \hspace*{-2cm} \quad =\left(\f{3}{4}-\f{\alpha}{2}\right)x^{\alpha-2}+\left(\f{\alpha}{2}-\f{3}{4}\right)x^{\alpha-2}\sum_{j=1}^{N} \prod_{\ell=1}^{j}[\ln_{\ell}(x)]^{-1}   \no\\
&  \hspace*{-2cm} \qquad 
-\f{1}{4}x^{\alpha-2}\sum_{j=1}^{N}\prod_{\ell=1}^{j}[\ln_{\ell}(x)]^{-1}  
+\f{1}{4}x^{\alpha-2}\Bigg(\sum_{j=1}^{N}\prod_{\ell=1}^{j}[\ln_{\ell}(x)]^{-1}\Bigg)^2   \no \\
&  \hspace*{-2cm} \qquad + \f{1}{2}
x^{\alpha-2}\sum_{j=1}^{N}\prod_{\ell=1}^{j}[\ln_{\ell}(x)]^{-1}\sum_{m=1}^{j}\prod_{p=1}^{m}[\ln_{p}(x)]^{-1}.    \lb{A.6}
\end{align}
Using the equality
\begin{align} \no
\sum_{j=1}^{N}\prod_{\ell=1}^{j}[\ln_{\ell}(x)]^{-1}\sum_{m=1}^{j}\prod_{p=1}^{m}[\ln_{p}(x)]^{-1}=&
\sum_{j=1}^{N-1}\prod_{\ell=1}^{j}[\ln_{\ell}(x)]^{-2}\sum_{m=j+1}^{N} 
\prod_{p=j+1}^{m}[\ln_{p}(x)]^{-1}     \lb{A.7} \\ 
& +\sum_{j=1}^{N}\prod_{\ell=1}^{j}[\ln_{\ell}(x)]^{-2},
\end{align}
one rewrites the last line in \eqref{A.6} in the form
\begin{align} 
& \f{1}{4}x^{\alpha-2}\Bigg(\sum_{j=1}^{N}\prod_{\ell=1}^{j}[\ln_{\ell}(x)]^{-1}\Bigg)^{2}+\f{1}{2}x^{\alpha-2}\sum_{j=1}^{N}\prod_{\ell=1}^{j}[\ln_{\ell}(x)]^{-1}\sum_{m=1}^{j}\prod_{p=1}^{m}[\ln_{p}(x)]^{-1} \no\\
& \quad =\f{1}{4}x^{\alpha-2}\sum_{j=1}^{N-1}\prod_{\ell=1}^{j}[\ln_{\ell}(x)]^{-1}\sum_{m=j+1}^{N}\prod_{p=1}^{m}[\ln_{p}(x)]^{-1}      \no \\ 
& \qquad
+\f{3}{4}x^{\alpha-2}\sum_{j=1}^{N}\prod_{\ell=1}^{j}[\ln_{\ell}(x)]^{-1}\sum_{m=1}^{j}
\prod_{p=1}^{m}[\ln_{p}(x)]^{-1}     \lb{A.8} \\
& \quad =x^{\alpha-2}\sum_{j=1}^{N-1}\prod_{\ell=1}^{j}[\ln_{\ell}(x)]^{-2}\sum_{m=j+1}^{N}\prod_{p=j+1}^{m}[\ln_{p}(x)]^{-1} +\f{3}{4}x^{\alpha-2}\sum_{j=1}^{N}\prod_{\ell=1}^{j}[\ln_{\ell}(x)]^{-2}.    \no 
\end{align}
Taking into account \eqref{3.11}, \eqref{A.6} and \eqref{A.8}, one derives the result 
\begin{align}
\f{1}{y_{N}(x)}\big(x^{\alpha}y_{N}'(x)\big)'=& [(3/4)-(\alpha/2)] 
x^{\alpha-2} - (1/2) (2- \alpha) x^{\alpha-2}\sum_{j=1}^{N}\prod_{\ell=1}^{j}[\ln_{\ell}(x)]^{-1}     \no\\ 
& + (3/4) x^{\alpha-2}\sum_{j=1}^{N}\prod_{\ell=1}^{j}[\ln_{\ell}(x)]^{-2} = q_{\al, N}(x)   \no \\
& + x^{\alpha-2}\sum_{j=1}^{N-1}\prod_{\ell=1}^{j}[\ln_{\ell}(x)]^{-2}\sum_{m=j+1}^{N}\prod_{p=j+1}^{m}[\ln_{p}(x)]^{-1}. 
    \lb{A.9}
\end{align}
\end{proof}
%%%%%%

In connection with the limit circle case discussed in Theorem \ref{t3.3} we note the following result:

%%%%%%
\begin{lemma} \lb{lA.2}
Let the conditions of Theorem \ref{t3.3} be satisfied, and $q_{\alpha,N,\varepsilon}(x)$, $\tau_{\alpha,N,\varepsilon}$,  $y_{N,\varepsilon}(x)$, and
$\wti y_{N,\varepsilon}(x)$ be as in \eqref{3.34}--\eqref{3.37}.  Then for all $N\in\bbN$,
\begin{align}
(\tau_{\alpha,N,\varepsilon}\,y_{N,\varepsilon})(x)=0,\quad (\tau_{\alpha,N,\varepsilon}\,\wti y_{N,\varepsilon})(x)=0.
	\lb{A.10}
\end{align}
\end{lemma}
%%%%%%
\begin{proof} Since $y_{N,\eps}(x)=y_N(x)[\ln_{N}(x)]^{-\varepsilon/2}$, one derives
\begin{align}\no
\big(x^{\al}y'_{N,\eps}(x)\big)'=&\big(x^{\al}y'_{N}(x)\big)' [\ln_{N}(x)]^{-\varepsilon/2}+2x^{\al}y'_{N}(x)\big([\ln_{N}(x)]^{-\varepsilon/2}\big)'
\\ &
+x^{\al}y_{N}(x)\big([\ln_{N}(x)]^{-\varepsilon/2}\big)''+\al x^{\al-1}y_{N}(x)\big([\ln_{N}(x)]^{-\varepsilon/2}\big)'.
	\lb{A.11}
\end{align}
Employing \eqref{A.2}, \eqref{A.4}, and \eqref{A.5}, we next calculate several derivatives, which appear on the right-hand side of \eqref{A.11}:
\begin{align}
& \lb{A.13}
\big([\ln_{N}(x)]^{-\varepsilon/2}\big)'= (\eps/2) x^{-1}[\ln_{N}(x)]^{-\varepsilon/2}\prod_{k=1}^{N}[\ln_{k}(x)]^{-1},
\\ &
\no
\big([\ln_{N}(x)]^{-\varepsilon/2}\big)''=- (\eps/2) x^{-2}[\ln_{N}(x)]^{-\varepsilon/2}\prod_{k=1}^{N}[\ln_{k}(x)]^{-1}
\\ \lb{A.14} &\qquad\qquad\qquad\qquad
+ \big(\eps^2\big/4\big) x^{-2} [\ln_{N}(x)]^{-\varepsilon/2}\prod_{k=1}^{N}[\ln_{k}(x)]^{-2}
\\ \no &\qquad\qquad\qquad\qquad
+ (\eps/2) x^{-2} [\ln_{N}(x)]^{-\varepsilon/2}\prod_{k=1}^{N}[\ln_{k}(x)]^{-1} 
\sum_{j=1}^{N}\prod_{\ell=1}^{j}[\ln_{\ell}(x)]^{-1},
\\ \lb{A.15} &
\f{y'_{N}(x)}{y_{N}(x)}= (2 x)^{-1}\Bigg[-1+ \sum_{j=1}^{N}\prod_{\ell=1}^{j}[\ln_{\ell}(x)]^{-1}\Bigg].
\end{align}

It follows from the last equality in \eqref{A.9}, from \eqref{A.11},  from \eqref{A.13}--\eqref{A.15}, and from \eqref{3.34} that
\begin{align}\no
& \f{1}{y_{N,\varepsilon}(x)}\big(x^{\alpha}y'_{N,\varepsilon}(x)\big)' 
= q_{\al,N}(x)- (\eps/2) x^{\al-2}\prod_{k=1}^{N}[\ln_{k}(x)]^{-1}
\Bigg(1-\sum_{j=1}^{N}\prod_{\ell=1}^{j}[\ln_{\ell}(x)]^{-1}\Bigg)   \no \\ 
& \qquad + (\eps/2) x^{\al-2} \Bigg(-\prod_{k=1}^{N}[\ln_{k}(x)]^{-1}
+ (\eps/2)\prod_{k=1}^{N}[\ln_{k}(x)]^{-2}     \no \\ 
& \qquad +\prod_{k=1}^{N}[\ln_{k}(x)]^{-1}\sum_{j=1}^{N}\prod_{\ell=1}^{j}[\ln_{\ell}(x)]^{-1}\Bigg) 
+ (\eps/2) \al x^{\al-2}\prod_{k=1}^{N}[\ln_{k}(x)]^{-1}.     \no \\
 & \quad = q_{\al,N}(x) - (\eps/2) (2 - \al) x^{\al-2}\prod_{k=1}^{N}[\ln_{k}(x)]^{-1}
+ \big(\eps^2\big/4\big) x^{\al-2}\prod_{k=1}^{N}[\ln_{k}(x)]^{-2}    \no \\
& \qquad +\eps x^{\al-2}\prod_{k=1}^{N}[\ln_{k}(x)]^{-1}\sum_{j=1}^{N}\prod_{\ell=1}^{j}[\ln_{\ell}(x)]^{-1}  \no \\ 
& \quad = q_{\al,N,\eps} (x).      \lb{A.16} 
\end{align}
Thus, the first equality in \eqref{A.10} is proved, and the second equality in \eqref{A.10}
is clear from the reduction of order approach. 
\end{proof}
%%%%%%

\medskip

%%%%%%%%%%%%%%%%%%%%%%%%%%%%%%%%%%%%%
\noindent
{\bf Acknowledgments.} 
We gratefully acknowledge discussions with Tony Zettl. We also thank the referee for a very careful reading of our manuscript and for his constructive comments. 
F.G. is indebted to Alexander Sakhnovich for the great hospitality extended to him at the Faculty of Mathematics of the University of Vienna, Austria, during an extended stay in June of 2019. A.S. is very grateful for the kind hospitality he experienced at Baylor University in October of 2019. The research of A.S. was supported by the Austrian Science Fund (FWF) under Grant No.~P29177. 
%%%%%%%%%%%%%%%%%%%%%%%%%%%%%%%%%%%%%

%%%%%%%%%%%%%%%%
%%%%%%%%%%%%%%%%

\end{document}